\documentclass[11pt]{article}
\usepackage{amssymb,amsmath}

\usepackage{latexsym}

\newfont{\bb}{msbm10 at 11pt}
\newfont{\bbsmall}{msbm8 at 8pt}

\usepackage[margin=1.5in]{geometry}

\newcommand{\h}{\widehat}
\newcommand{\wt}{\widetilde}

\newcommand{\R}{\mbox{\bb R}}

\newcommand{\N}{\mbox{\bb N}}
\newcommand{\HH}{\mbox{\bb H}}

\newcommand{\Nsmall}{\mbox{\bbsmall N}}

\newcommand{\rth}{\R^3}

\newcommand{\ben}{\begin{enumerate}}
\newcommand{\bit}{\begin{itemize}}
\newcommand{\een}{\end{enumerate}}
\newcommand{\eit}{\end{itemize}}

\def\a{{\alpha}}
\def\lc{{\cal L}}

\def\g{{\gamma}}

\def\de{{\delta}}

\def\ve{{\varepsilon}}

\def\centerbmp#1#2#3{\vskip#2\relax\centerline{\hbox to#1{\special
    {bmp:#3 x=#1, y=#2}\hfil}}}

\newtheorem{theorem}{Theorem}[section]
\newtheorem{lemma}[theorem]{Lemma}
\newtheorem{proposition}[theorem]{Proposition}

\newtheorem{remark}[theorem]{Remark}
\newtheorem{corollary}[theorem]{Corollary}
\newtheorem{definition}[theorem]{Definition}

\newtheorem{assertion}[theorem]{Assertion}

\newenvironment{proof}{\smallskip\noindent{\it Proof.}\hskip \labelsep}
                          {\hfill\penalty10000\raisebox{-.09em}{$\Box$}\par\medskip}

\begin{document}
\begin{title}
{Existence of regular neighborhoods for $H$-surfaces }
\end{title}
\begin{author}
{William H. Meeks, III\thanks{ This material is based upon
 work for the NSF under Award No. DMS -
 0703213. Any opinions, findings, and conclusions or recommendations
 expressed in this publication are those of the authors and do not
 necessarily reflect the views of the NSF.} \and Giuseppe Tinaglia\thanks{Partially
supported by The Leverhulme Trust.}}
\end{author}
\maketitle

\begin{abstract} In this paper, we study the global geometry of
complete,  constant mean curvature hypersurfaces embedded in
$n$-manifolds. More precisely, we give conditions that imply properness
of such surfaces and prove the existence of fixed size one-sided regular neighborhoods for certain constant mean curvature hypersurfaces in certain $n$-manifolds.

\vspace{.3cm}

\noindent{\it Mathematics Subject Classification:} Primary 53A10,
   Secondary 49Q05, 53C42

\noindent{\it Key words and phrases:}
Minimal surface, constant mean curvature, regular neighborhood, properly embedded.
\end{abstract}

\section{Introduction.}

In this paper we present some useful results on the geometry of a complete $H$-hypersurface $M$ in an $n$-manifold $N$, where by $H$-hypersurface we mean that $M$ is an {\it embedded} (injectively immersed) hypersurface of constant mean curvature $H\geq 0$. When $H>0$, $N$ has bounded absolute sectional curvature and $M$ is connected, has bounded second fundamental form and it is proper and separating in $N$, then we prove that $M$ has a fixed size regular neighborhood on its mean convex side (see Remark~\ref{remark} and Theorem~\ref{cor*}). This result is useful for obtaining local $(n-1)$-dimensional volume estimates for such an $M$ in terms of local volume estimates in $N$. This existence of a fixed size one-sided neighborhood for certain $H$-surfaces $M$ is a cornerstone in proving the $CMC$ Dynamics and Minimal Elements Theorems in~\cite{mt4}, and more generally, the $CMC$ Dynamics and Minimal Elements Theorems in other homogenous $n$-manifolds such as $\R^n$ and hyperbolic $n$-space $\HH^n$~\cite{mt5}.
Theorem~\ref{cor*} also plays a key role in the classification of $H$-spheres in the 3-dimensional solvable group ${\rm Sol}_3$ with respect to its usual homogeneous structure; see \cite{me34}.

In Proposition~\ref{proper}, we demonstrate that complete embedded $H$-hypersurfaces $M$ with locally bounded second fundamental form in certain complete $n$-manifolds $N$ are properly embedded. For example, when $n=3$ this proposition implies that if for some $\ve>0$, $M$ and $N$ satisfy  $H^2\geq S_N+\ve$, where $S_N$ is the average sectional curvature (scalar curvature) of $N$, then $M$ is properly embedded in $N$; this special case of Proposition~\ref{proper} was found independently by Harold Rosenberg and this case is essentially the statement of Theorem~1.1 in~\cite{rose5} (also see \cite{mpr19}). If $N=\HH^3$ and $H>1$, it then follows that such an $M$ is properly embedded in $\HH^3$; see \cite{mt11} for the same properness result for connected, embedded $H=1$ surfaces in $\HH^3$.


Finally, in the last section we observe how using results by Grosse-Brauckmann in~\cite{gb1} make it possible to construct  examples of complete disconnected $H$-surfaces in a fixed horizontal slab in $\rth$, each of  which is properly embedded in the open slab but not properly embedded in $\rth$.

\section {Background on $H$-laminations.}\label{sec2}

In order to help understand the results described in this paper, we
make the following definitions. For further discussion on the
general theory of weak $H$-laminations, whose definition appears
below, and related $CMC$ foliations, see the survey~\cite{mpr19}.
\begin{definition} {\rm
Let $M$ be a complete surface embedded  in a three-manifold $N$. A
point $p\in N$ is a {\it limit point} of $M$ if there exists a
sequence of points $\{p_n\}_n\subset M$ which diverges to infinity
on $M$ with respect to the intrinsic Riemannian topology on $M$ but
converges in $N$ to $p$ as $n\to \infty$. Let $L(M)$ denote the set
of all limit points of $M$ in $N$. In particular, $L(M)$ is a closed
subset of $N$ and $\overline{M} -M \subset L(M)$, where
$\overline{M}$ denotes the closure of $M$ in $N$.}
\end{definition}

\begin{definition}\label{definition}
{\rm A {\em weak} $H$-{\em lamination} $\cal{L}$ of a three-manifold $N$
is a collection of immersed surfaces $\{L_\alpha\}_{\alpha\in I}$ of
constant positive mean curvature $H$ called {\em leaves} of
$\cal{L}$ satisfying the following properties.
\begin{enumerate}
\item ${\cal L}=\bigcup_{\alpha\in I} L_\alpha $ is a closed subset of
$N$.
\item For each leaf $L_\alpha$ of $\cal{L}$, considered to be the zero section $Z_\alpha$
of its normal bundle $N_b(L_\alpha)$, there exists a closed one-sided
neighborhood $\Delta(Z_\alpha)\subset N_b(L_\alpha)$ of $Z_\alpha$ such that:
\begin{enumerate}
\item The exponential map $\exp \colon \Delta(Z_\alpha) \rightarrow N$ is a
submersion.
\item With respect to the pull-back metric on $\Delta(Z_\alpha)$,
$Z_\alpha \subset \partial \Delta(Z_\alpha)$ is mean convex.
\item The inverse image $\exp^{-1}({\cal L})\cap \Delta(Z_\alpha)$ is a lamination of
$\Delta(Z_\alpha)$.

\end{enumerate}
\end{enumerate}\noindent When $H=0$,  by the maximum principle, each leaf of ${\cal L}$ is embedded; thus $\lc$ is an actual lamination and we call $\lc$ a {\it minimal}
lamination. When $H>0$, the maximum principle implies that each non-embedded leaf of a weak $H$-lamination is almost-embedded in the following sense.}
\end{definition}

\begin{definition}{\rm
We call an immersed hypersurface in an $n$ manifold $N$, which is the image of an immersion $f\colon M\to N$, {\em almost-embedded} if whenever  $p,q\in M$ and $f(p)=f(q)$, then there exist disjoint neighborhoods $U_p, U_q$ such that $f(U_p)$ and $f(U_q)$ lie locally on opposite sides of $f(p)$ and $f$ is injective on an open dense set of $M$. }

\end{definition}

The reader not familiar with the subject of minimal or weak
$H$-laminations should think about a geodesic on a Riemannian
surface. If the geodesic is complete and embedded (a one-to-one
immersion), then its closure is a geodesic lamination of the
surface. When this geodesic has no accumulation points, then it is
proper. Otherwise, there pass complete embedded geodesics through
the accumulation points forming the leaves of the geodesic
lamination of the surface. The similar result is true for a
complete $H$-surface of locally  bounded second
fundamental form (curvature is bounded in compact extrinsic balls)
embedded in a Riemannian three-manifold, i.e., the closure of a complete,
embedded $H$-surface of locally bounded second fundamental form has
the structure of a weak $H$-lamination. For the sake of
completeness, we now sketch the proof of this elementary fact for $H$-surfaces in 3-manifolds,
see~\cite{mr8} for the proof in the minimal case.

Consider a complete $H$-surface $M$ of locally bounded
second fundamental form embedded in a three-manifold $N$. Consider a limit
point $p$ of $M$, which is a limit of a sequence of divergent
points $p_n$ in $M$. Since $M$ has bounded curvature near $p$ and
$M$ is embedded, then for some small $\ve >0$, a subsequence of the
intrinsic disks $B_M(p_n, \ve)$ converges $C^2$ to an embedded
$H$-disk $D(p,\ve)\subset N$ of intrinsic radius $\ve$, centered at
$p$ and of constant mean curvature $H$. Since $M$ is embedded, any
two such limit disks, say $D(p,\ve)$, $D'(p,\ve)$, do not intersect
transversally (in fact, two such intersecting disks lie
locally on one side of each other). By  the maximum principle for $H$-surfaces, we
conclude that if a second disk $D'(p,\ve)$ exists, then $D(p,\ve)$,
$D'(p,\ve)$ are the only such limit disks and they are oppositely
oriented at $p$.

Now consider any sequence of embedded disks $E_n$ of the form
$B_M(q_n,\frac{\ve}{4})$ or $D(q_n,\frac{\ve}{4})$ such that $q_n$
converges to a point in $D(p,\frac{\ve}{2})$ and such that $E_n$
locally lies on the mean convex side of $D(p,\ve)$. For $\ve$
sufficiently small and for $n$, $m$ large, $E_n$ and $E_m$ must be
graphs over domains in $D(p,\ve)$ such that when oriented as graphs,
they have the same mean curvature (see also the
proof of Lemma~\ref{thm1}). By the maximum principle,
the graphs $E_n$ and $E_m$ are disjoint. It follows that near $p$
and on the mean convex side of $D(p,\ve)$, $\overline{M}$ has the
structure of a lamination with leaves with constant mean curvature
$H$. This proves that $\overline{M}$ has the structure of a weak
$H$-lamination.

The next proposition is a simple consequence of the fact that the closure $\overline{M}$ of a complete, $H$-hypersurface $M$ of locally bounded second fundamental form in a complete $n$-manifold $N$ has the structure of a weak  $H$-lamination and the fact that if $M$ is not proper in $N$, then $\overline{M}$ contains a limit leaf $L$ which is stable by the main theorem in~\cite{mpr18}.

\begin{proposition} \label{proper} Suppose $M$ is a complete connected $H$-hypersurface with locally bounded second fundamental form in a complete $n$-manifold $N$. If the only complete, stable, almost-embedded, $H$-hypersurfaces in $N$ are compact with finite fundamental group, then $M$ is properly embedded in $N$.
\end{proposition}

\begin{proof} After possibly lifting to a 2-sheeted cover of $N$, we may assume that $M$ is 2-sided. If $M$ is not proper in $N$, then ${\cal L}=\overline{M}$ has the structure of a weak $H$-lamination with a limit leaf $L$. By the Theorem 1 in~\cite{mpr18}, $L$ is stable and since it is almost-embedded, our hypotheses imply that $L$ is compact with finite fundamental group.

Let $\Pi\colon \wt{L}\to L$ denote the finite sheeted universal cover of $L$ and consider $\wt{L}$ to be the zero section of its normal bundle $N_b(\wt{L})$. For some small $\ve>0$, the $\ve$-neighborhood $\Delta(\wt{L})$ in $N_b(\wt{L})$ submerses to its image in $N$ under the exponential map. With respect to the pulled back lamination $\wt{\cal L}=(\exp|_{\Delta(\wt{L})})^{-1}({\cal L})$, we obtain a lamination of $\Delta(\wt{L})$ with the compact simply-connected hypersurface $\wt{L}$ as a limit leaf. By a monodromy argument, the leaves in $\wt{\cal L}$ close to $\wt{L}$ are also compact and naturally diffeomorphic to $\wt{L}$, and there is a sequence of them which converges to $\wt{L}$. Hence, the images or projections to $N$ under the exponential map of these compact leaves correspond to components of $M$, which contradicts our assumption that $M$ is connected. This contradiction completes the proof of the proposition.
\end{proof}

The importance of the above proposition is that in certain cases, the hypotheses in the statement of it follow from simple geometric hypotheses. In the next corollary we list some of these cases.  Item 1 in the next corollary was found independently by Harold Rosenberg and this item is essentially the statement of Theorem~1.1 in~\cite{rose5} (also see \cite{mpr19}).

\begin{corollary} Let $N$ be an $n$-manifold with absolute sectional curvature less than or equal to 1 and let $M$ be a 2-sided $H$-hypersurface  with locally bounded second fundamental form in $N$. If any of the following items holds, then $M$ is properly embedded in $N$.
\ben
\item $n=3$ and for some $\ve>0$, $H^2\geq -S_N+\ve $, where $S_N$ is the average sectional curvature of $N$ (also called the scalar curvature of $N$). Note that if $S_N\geq \delta >0$, then this condition holds for all $H\geq 0$.
\item $n=4$ and $H>\frac{10}{9}$.
\item $n=5$ and $H>\frac{7}{4}$.
\item $n=3,4,5$, $N$ is flat and $H>0$.
\een
\end{corollary}
\begin{proof} Any of the above listed conditions is  sufficient to guarantee that any stable $H$-hypersurface  in $N$ is compact with finite fundamental group. The condition in item 1 appears in~\cite{enr1,lor2,rose4} and the idea behind it originates from the techniques developed by Fischer-Colbrie and Schoen in~\cite{fs1} (see also~\cite{mpr19} for a complete discussion). See~\cite{chen1,enr1} for the conditions in items 2 and 3 and see \cite{chen1,enr1,lor2,rose4} for item 4.

 Applying Proposition~\ref{proper} completes the proof of the corollary.
\end{proof}

\section{The existence of a one-sided regular neighborhood.}\label{area}
In~\cite{mr7} Meeks and Rosenberg proved that a complete $H$-surface $M$ of bounded
second fundamental form and embedded in $\rth$
has a fixed size
regular neighborhood on its mean convex side (on both sides when $H=0$). In other words, for
such an $M$ there exists an $\ve>0$ such that for any $p\in M$, the
normal line segment $l_p$ of length $\ve$ based at $p$ and contained
on the mean convex side of $M$, forms a collection of pairwise disjoint embedded line segments. An immediate consequence of the existence of such a regular
neighborhood is that the surface is properly embedded and for some
$c>0$, the area of the surface is at most $cR^3$ in any ball of
radius $R\geq 1$.  The next lemma and theorem generalize this result in the case
that the mean curvature is positive.

\begin{lemma}
\label{thm1} Suppose $N$ is a complete $n$-manifold with absolute
sectional curvature bounded by a constant $S_0$ and with injectivity
radius at least $I_0>0$. Suppose $W\subset N$ is a proper smooth mean convex subdomain whose boundary $M$ has constant mean
curvature $H_0 >0$ and  $|A_M|\leq A_0$, where $|A_M|$ is the norm of the second fundamental form of $M$. Then
there exists a positive number $\tau\in(0,I_0)$, depending on $A_0, H_0, I_0, S_0,$ such
that $M$ has a regular neighborhood of width $\tau$ in $W$.
\end{lemma}

\begin{remark} {\rm
Since a geodesic triangle $W$ in the hyperbolic plane with 3 vertices at infinity does not have a fixed size regular neighborhood in $W$ for its boundary,  we see that Lemma~\ref{thm1} does not generalize to the case when the mean curvature vanishes.}
\end{remark}

\begin{proof}
The uniform
bound $|A_M|\leq A_0$ together with the bounds given by $I_0$ and
$S_0$ implies that there exists an $\ve>0$ sufficiently small so that
for any $p\in M$, every component of $B_{N}(p,2\ve)\cap M$ is a
graph over its projection to the disk of radius $3\ve$ in $T_pM$;
here we are considering the tangent plane to be a plane in normal
coordinates and for the orthogonal projection map to be well defined
in these coordinates. Moreover, we can choose $\ve$ sufficiently
small and a smaller positive $\de$, depending on the bounds $A_0,
I_0, S_0$, such that every component of $B_{N}(p,2\ve)\cap M$
intersecting $B_{N}(p,\de)$ contains a graph over $D(p,\ve)$,
that is the disk of radius $\ve$ in $T_pM$. Also, we
may assume that in our coordinates, these graphs all have
gradient of norm less than 1.

The theorem will follow from the observation that two disjoint
graphs, $u_1$ and $u_2$ over $D(p,\ve)$ with bounded gradients and of constant mean curvature
$H_0$ which are oppositely oriented and such that $u_2$ lies on the
mean convex side of $u_1$, cannot be too close at their centers (this is essentially a consequence of the maximum principle for quasi-linear uniformly elliptic PDEs). Let
$p\in M$ and let $u_1\subset M$ be the graph over $D(p,\ve)$
containing $p$. Let $\de_1\in(0,\de)$ and suppose that
$B_N(p,\de_1)\cap M$ contains a connected component different from
$u_1\cap B_N(p,\de_1)$ and which lies on the mean convex side of
$u_1$. Let $u_2\subset M$ be the graph over $D(p,\ve)$ which
contains such a component and choose it to be the closest one.
In other words, the region between the graphs is contained in $W$ and so,
since $W$ is mean convex, $u_2$ lies on the
mean convex side of $u_1$ and they are oppositely oriented.
This contradicts our observation, if $\de_1$ is chosen sufficiently small.
 Hence, when
$\de_1$ is sufficiently small, $B_N(p,\de_1)$ does not intersect $M$
on the mean convex side of $u_1$. Letting $\tau$ be such a small
$\de_1$ proves the lemma.
\end{proof}


\begin{definition}{\rm
 We call a proper, almost-embedded $H$-hypersurface $f\colon M\to N$ {\em strongly Alexandrov embedded} if there exists a proper immersion $F\colon W\to N$ of a mean convex 3-manifold $W$ with $\partial W = M$, which extends the map $f$ and which is injective on the interior of $W$. }

\end{definition}

\begin{remark} \label{remark} {\rm Note that if $M$ is a connected, properly embedded and separating $H$-hypersurface in an $n$-manifold $N$, then it is strongly Alexandrov embedded.   }

\end{remark}

\begin{theorem}[One-sided Regular Neighborhood]
\label{cor*}  Suppose $N$ is a complete $n$-manifold with absolute
sectional curvature bounded by a constant $S_0$. Let $M$ be a strongly Alexandrov embedded hypersurface with constant mean
curvature $H_0 >0$ and  $|A_M|\leq A_0$. Then the following
statements hold.
\begin{enumerate}
\item There exists a positive number $\tau\in\left(0,\frac{\pi}{S_0}\right)$, depending on $A_0$, $H_0$, $S_0,$ such
that $M$ has a regular neighborhood of width $\tau$ on its mean convex side.
\item There exists a positive number $C$, depending on $A_0$, $H_0$, $S_0,$
such that  the $(n-1)$-dimensional volume of $M$ in balls of
radius $1$ in $N$ is less than $C$. \end{enumerate}
\end{theorem}

\begin{proof} We will first prove the existence of a one-sided regular
neighborhood of $M$ in $N$ with width a certain $\tau\in\left(0,\frac{\pi}{S_0}\right)$. We will prove item 1 in the case when $M$ is embedded; with minor modifications, the same arguments demonstrate the case when $M$ is almost-embedded.

Let $\Pi\colon \wt{N}\to N$ be the universal cover of $N$. Since $M$ is strongly Alexandrov embedded, there is a proper subdomain $W\subset N$ with $\partial W=M$ that lies on the mean convex side of $M$. Without loss of generality, we can assume that $W$ is connected. Let $\wt{W}\subset \wt{N}$ be one of the connected components of $\Pi^{-1}(W)$. Note that $\partial \wt{W}$ is mean convex with respect to the pull-back metric. Furthermore, it consists  of a collection of surfaces $\{\wt{M}_{\a}\}_{\a\in I}$ such that for each $\a\in I$, the restriction $\Pi\colon \wt{M}_{\a}\to M$ is a covering space.

Recall that any complete simply-connected manifold with sectional curvature bounded from above by $S_0\geq 0$ has positive injectivity radius at least $\frac{\pi}{S_0}$. Therefore, by Lemma~\ref{thm1}, there exists a one-sided regular $\tau$-neighborhood, $N(\partial \wt{W},\tau)$ for $\partial \wt{W}$, where $\tau\in\left(0,\frac{\pi}{S_0}\right)$ depends on $A_0, H_0$ and $S_0$. Note that for any point $p\in N(\partial \wt{W},\tau)-\partial  \wt{W}$, there exists a unique geodesic in $N(\partial \wt{W}, \tau)$ joining $p$ to $\partial \wt{W} $ of least length.

For each $x\in \partial \wt{W}$, let $\g_x \subset \wt{W}$ be the unit speed geodesic starting at $x$, perpendicular to $\partial \wt{W}$ and parameterized on $[0,\tau)$.

\begin{assertion} \label{assertion}
For each  $x, y \in \partial \wt{W}$ the following hold.
\ben
\item $\Pi (\g_x)\cap M=\Pi (x)$ and $\Pi(\g_x(0,\tau))\cap M = \O$.
\item  $\Pi |_{\g_{x}}$ is injective.
\item  If $\Pi(x)\not=\Pi(y)$, then $\Pi(\g_x)\cap \Pi(\g_y)=\O.$
\een

\end{assertion}

\begin{proof} Item 1 is a consequence of the fact that $\g_x(0,\tau)\subset {\rm Int}(\wt{W})$ and $\Pi({\rm Int}(\wt{W}))={\rm Int}(W)$.

In order to prove items 2 and 3, we first note that if $g\colon\wt{N}\to \wt{N}$ is a covering transformation and $g(\wt{W})\cap\wt{W}\not=\O$, then $g(\wt{W})=\wt{W}$. To see this suppose that $p\in g(\wt{W})\cap \wt{W}$. Since $g\colon \wt{N}\to \wt{N}$ is an isometry, $g\colon \Pi^{-1}(M)\to \Pi^{-1}(M)$ and $\partial \wt{W}$ is mean convex, we may pick $p$ to be in the interior of $\wt{W}\cap g(\wt{W})$. We first show that $\wt{W}\subset g(\wt{W})$. Let $q\in \wt{W}$ and choose a path $\g\subset \wt{W}$ from $p$ to $q$ and such that $\g-q\subset {\rm Int}(\wt{W})$. Note that $\g$ can only intersect $\Pi^{-1}(M)$ at $q$ and only if $q\in \partial \wt{W}$; hence $\g-q$ is disjoint from $\partial (g(\wt{W}))$. Since $p$ is in the interior of $g(\wt{W})$ and $\g-q$ is disjoint from $\partial (g(\wt{W}))$, then $\gamma \subset g(\wt{W})$ and thus $q\in g(\wt{W})$, which proves that $\wt{W}\subset g(\wt{W})$. Note that proving $g(\wt{W})\subset \wt{W}$ is equivalent to proving that $\wt{W}\subset g^{-1}(\wt{W})$. Applying $g^{-1}$ to $g(\wt{W})\cap \wt{W}\not=\O$ gives $\wt{W}\cap g^{-1}(\wt{W})\not=\O$. Arguing as previously, $\wt{W}\subset g^{-1}(\wt{W})$ and this proves that $\wt{W}=g(\wt{W})$.

We  now prove item 2. Suppose that there exists $x\in \partial \wt{W}$ for which $\Pi|_{\g_{x}}$ is not injective. Then there exist points $a,b\in \g_x$, $a\not= b$, such that $\Pi(a)=\Pi(b)$. Since $\wt{N}$ is the universal covering space of $N$, there exists a nontrivial covering transformation $g\colon \wt{N}\to \wt{N}$ such that $g(a)=b$. In particular,  $g(\wt{W})\cap \wt{W}\not=\O$ and by the previous observation, $g(\wt{W})=\wt{W}$. In other words $g|_{\wt{W}}\colon \wt{W}\to \wt{W}$ is an isometry and $$d_{\wt{W}}\left(a, \partial \wt{W}\right)=d_{\wt{W}}\left(g(a), \partial \wt{W}\right)=d_{\wt{W}}\left(b, \partial \wt{W}\right).$$ However, different points in $\g_x$ have different distances to $\partial \wt{W}$ and thus $$d_{\wt{W}}\left(a, \partial \wt{W}\right)\not= d_{\wt{W}}\left(b, \partial \wt{W}\right).$$ This contradiction finishes the proof of item 2.

The proof of item 3 is similar. Let $x,y\in \partial \wt{W}$ such that $\Pi(x)\not=\Pi(y)$. If $\, \Pi(\g_x)\cap \Pi(\g_y)\not=\O$, then there exist $a\in \g_x$ and $b\in \g_y$ such that $\Pi(a)=\Pi(b)$. As before, there exists a nontrivial covering transformation $g\colon \wt{N}\to \wt{N}$ inducing an isometry $g|_{\wt{W}}\colon \wt{W}\to \wt{W}$ and such that $g(a)=b$. Since $g|_{\wt{W}}$ preserves distances to $\partial \wt{W}$ and $g(a)=b$, then $g(\g_x)=\g_y$ and so $g(x)=y$, which implies $\Pi(x)=\Pi(y)$. This contradiction finishes the proof of item 3 and the proof of the assertion.
\end{proof}

It follows from Assertion \ref{assertion} that as $x\in M$ varies, the geodesic segments  $\g_x\subset W$ starting at $x$ and perpendicular to $M$ of length $\tau$ are embedded and form a pairwise disjoint collection. This proves the first statement of the theorem.

Let $N(M,\tau)$ be the one-sided regular neighborhood of $M$ in $N$
given by item {1}. For a domain $E\subset M$, let $N(E,\tau)\subset
N(M,\tau)$ be the associated one-sided regular neighborhood in $W\subset N$. The uniform bound on $|A_M|$ and on the absolute
sectional curvature of $N$  implies that there exists a constant $K$, depending also on $\tau$, such that the $(n-1)$-dimensional
volume of any compact domain $E$ on $M$ is less than $K$ times the
volume of $N(E,\tau)$. Thus, since the volumes of balls in $N$ of radius
$2$ are uniformly bounded by a constant depending only on $S_0$, the $(n-1)$-dimensional volume of $M$ in
such balls is also uniformly bounded. In other words, for such a $\tau$ and for any $p\in M$,
\[
\begin{split}
{\rm Volume}_{n-1}[M\cap B_N(p,1)]&\leq K \cdot {\rm Volume}_n [N(M\cap B_N (p,1),\tau)]\\&\leq K \cdot {\rm Volume}_n [B_N (p,2)]\leq C
\end{split}
\]
where $C$ depends solely $A_0$, $S_0$ and $H_0$. This proves item {2}, and thus
completes the proof of the theorem.
\end{proof}

Our proof of the Dynamics Theorem in~\cite{mt1}  uses the
following corollary to Theorem~\ref{cor*}.

\begin{corollary} Suppose $N$ is a complete, simply-connected $n$-manifold with absolute
sectional curvature bounded by a constant $S_0$. Let $M$ be a connected, properly immersed, almost-embedded hypersurface with constant mean
curvature $H_0 >0$ and  $|A_M|\leq A_0$. Then the following
statements hold.
\begin{enumerate}
\item There exists a positive number $\tau\in\left(0,\frac{\pi}{S_0}\right)$, depending on $A_0$, $H_0$, $S_0,$ such
that $M$ has a regular neighborhood of width $\tau$ on its mean convex side.
\item There exists a positive number $C$, depending on $A_0$, $H_0$, $S_0,$
such that  the $(n-1)$-dimensional volume of $M$ in balls of
radius $1$ in $N$ is less than $C$. \end{enumerate}
\end{corollary}

\begin{proof}
Since a connected, properly embedded surface in $N$ separates $N$ into two connected components, a surface $M$ satisfying the hypotheses of the corollary is seen to be strongly Alexandrov embedded. Applying Theorem~\ref{cor*} finishes the proof.
\end{proof}

\section{The existence of $H$-surfaces in ${\mathbb R}^3$ which are not properly embedded.}\label{slab}

In~\cite{gb1}, Grosse-Brauckmann  constructed a one-parameter family
$M(t)$, $0<t<\infty$, of properly embedded doubly periodic
$H$-surfaces with the $(x_1,x_2)$-plane $P$ as a plane of Alexandrov
symmetry and varying constant mean curvature $H_{M(t)}>0$. This
family contains the doubly periodic surface  $\xi_{2,2}$
of Lawson~\cite{la1}. The rescaled surfaces $\h{M}(t)=H_{M(t)}M(t)$
have constant mean curvature 1 and converge pointwise to $P$ with area
multiplicity two as $t\to 0$. In particular, for $n\in\N$, there exists a
properly embedded surface $M_n\subset \rth$ with constant mean
curvature $1$ and contained in the slab
$\{(\frac{1}{2})^{n+1}<x_3<(\frac{1}{2})^n\}$. The disconnected
surface $M_{\infty}=\bigcup_{n\in \Nsmall} M_n$ is a complete $1$-surface
properly embedded  in the slab $\{0<x_3<\frac{1}{2}\}$ which is
not properly embedded in $\rth$. A simple modification of this argument proves the following proposition when $H>0$ and when $H=0$, the proposition follows trivially by using an appropriate infinite collection of horizontal
planes in the slab.

\begin{proposition}
\label{nonproper} For any  $H\geq0$, there exists a complete,
disconnected $H$-surface properly embedded in the slab
$\{0<x_3<\frac{1}{2}\}$, which is not properly embedded in $\rth$.
\end{proposition}

\center{William H. Meeks, III at profmeeks@gmail.com}\\
Mathematics Department, University of Massachusetts, Amherst, MA
01003
\center{Giuseppe Tinaglia at giuseppe.tinaglia@kcl.ac.uk}\\ Department of Mathematics, King's College London, Strand, London, WC2R 2LS, U.K.

\bibliographystyle{plain}
\bibliography{bill}

\end{document}